\documentclass[12pt]{article}
\usepackage{amssymb,amscd,amsmath,amsthm,url}
\usepackage{mathtools}
\usepackage{enumerate}
\usepackage[all]{xy}
\usepackage{breqn}
\usepackage{lipsum} % just for the example
\def \R {\mathbb R}

\def \bj {\bar{j}}
\def \bl {\bar{l}}
\def \bm {\bar{m}}
\def \bn {\bar{n}}
\def \bk {\bar{k}}

\def \bi {\bar{i}}
\def \bq {\bar{q}}
\def \Ric {\text{Ric}}
\def \Rm {\text{Rm}}
\def \Re {\text{Re}}
\def \db {\bar{\partial}}
\def \d {\partial}
\def \Tr {\text{tr}}
\def \lapl {\Delta}

\newcommand{\dd}[1] {\frac{\partial}{\partial #1 }}

\newtheorem {Theorem} {Theorem}[section]
\newtheorem {Proposition}[Theorem]      {Proposition}
\newtheorem {Lemma}      [Theorem]       {Lemma}
\newtheorem {Corollary} {Corollary}[Theorem]

\theoremstyle{definition}

\theoremstyle{remark}

\title{The K\"ahler-Ricci flow on manifolds with negative holomorphic curvature}
\author{Freid Tong}
\date{}
\newcommand{\Addresses}{{% additional braces for segregating \footnotesize
  \bigskip
  \footnotesize
  \textsc{Department of Mathematics, Columbia University,
    2990 Broadway, New York, NY 10027}\par\nopagebreak
  \textit{E-mail address}: \texttt{tong@math.columbia.edu}
}}

\begin{document}
\maketitle
\begin{abstract}
We study the behaviour of the normalized K\"ahler-Ricci flow on complete K\"ahler manifolds of negative holomorphic sectional curvature. We show that the flow exists for all time and converges to a K\"ahler-Einstein metric of negative scalar curvature, recovering a result of Wu and Yau in \cite{WY3}. 
\end{abstract}
\section{Introduction}
The K\"ahler-Ricci flow has become one of the most important tools in understanding the existence of K\"ahler-Einstein metrics, see for example \cite{Cao, Chau} for well-known results in this direction. In 2015, Wu and Yau in \cite{WY1, WY2} proved the existence of K\"ahler-Einstein metrics with negative scalar curvature on compact K\"ahler manifold $X$ when $X$ admits a K\"ahler metric with negative holomorphic sectional curvature, their result sheds light on the deep relationship between holomorphic sectional curvature and Ricci curvature. Recently in \cite{WY3}, they extended their result to the case where $X$ is complete noncompact while assuming the sectional curvature of $X$ is bounded. The approach of using the K\"ahler-Ricci flow to study the same problem has been explored by Nomura in \cite{Nomura}, where he reproves the results in \cite{WY1, WY2} using the K\"ahler-Ricci flow. However, his method do not extend to the noncompact setting as he uses a result of Demailly-Paun on the numerical characterization of the K\"ahler cone \cite{DP}, which has no analogue in the noncompact setting. 

In this note, we will take the K\"ahler-Ricci flow approach to recover the results of \cite{WY3} by showing the convergence of the normalized K\"ahler-Ricci flow to a complete K\"ahler-Einstein metric. One advantage of this approach is that it relies only on elementary maximum principle and we are able avoid some of the more sophisticated arguments used in \cite{WY3}. For instance, we do not make use of the quasi-bounded coordinate charts constructed in \cite{WY3}, and we can also avoid the maximum principle argument used in \cite{WY3} to show the negative holomophic sectional curvature condition is preserved for short time along the Ricci flow. This is also interesting from the point of view of parabolic PDEs because the equation obtained \eqref{eq:CMA}, is a complex monge ampere equation with a background metric which is time dependent, hence the nonlinear parabolic theory on manifolds developed in \cite{PT} does not apply. 

\begin{Theorem}\label{main}
Let $(M, \omega_0)$, is a complete K\"ahler manifold with bounded curvature, that is
\[\sup_M|\Rm (\omega_0)|\leq B\]
for some constant $B>0$, and suppose that the holomorphic sectional curvature of $\omega_0$ is bounded above by a negative constant,
\[H_{\omega_0}(\eta) = R_{\omega_0}(\eta, \bar{\eta}, \eta, \bar{\eta})\leq -\kappa|\eta|^4\]
for some $\kappa>0$. 
Then the normalized K\"ahler-Ricci flow starting at $\omega_0$ exists for all time and converges smoothly to a complete K\"ahler-Einstein metric $\omega_{KE}$ of negative scalar curvature on $M$ which satisfies
\begin{enumerate}
\item $\Ric(\omega_{KE}) = -\omega_{KE}$
\item $C^{-1}\omega_0\leq \omega_{KE}\leq C\omega_0$
\item $|\nabla_{\omega_{KE}}^l\Rm(\omega_{KE})|_{\omega_{KE}}\leq C_l$
\end{enumerate}
for constants $C$ and $C_l$, for $l = 0, 1,\ldots$ depending only on $B$ and $\kappa$. 
\end{Theorem}

\section{Lower bound on the evolving metric}
The short time existence of the Ricci flow when $(M, \omega_0)$ is complete and has bounded curvature was established by Shi in \cite{Shi}. In this case, he is able to obtain a Ricci flow solution for short time which also has bounded curvature and satisfied his derivative estimates. The uniqueness of this solution is proved by Chen and Zhu in \cite{CZ}. 

From now on, let $(M, \omega(t))$ denote the solution of the normalized K\"ahler-Ricci flow with initial metric $\omega_0$ obtained from Shi's estimates. That is, $\omega(t)$ satisfies
\begin{align}
\dd{t}\omega(t) &= -\Ric(\omega(t))-\omega(t)\\
\omega(0)& = \omega_0
\end{align}
and the metrics $\omega(t)$ are all complete with bounded curvature, are uniformly equivalent to the initial metric $\omega_0$ and satisfies Shi's derivative estimates.

Before we proceed to the main content of this section, we need a version of the maximum principle on non-compact manifolds along the Ricci flow. This can be found for instance in \cite{Shi2}, but we will prove a version that we need here. 

\begin{Proposition}\label{max}
Suppose $(M, g(t))_{t\in[0, T]}$ is a complete solution of the K\"ahler-Ricci flow with bounded curvature. Then for any $C^2$ function $f$ which is bounded above on $M\times [0, T]$, either
\begin{enumerate}
\item $\sup_{M\times [0, T]}f(x, t) = \sup_{M}f(x, 0)$
\item There exist a sequence of points $(x_k, t_k)$, such that 
\[\lim_{k\to \infty}f(x_k, t_k)  = \sup_{M\times [0, T]} f \qquad\qquad \lim_{k\to \infty} |\nabla f|(x_k, t_k) = 0\]
\[(\dd{t}-\lapl)f(x_k, t_k) \geq -\epsilon_k\]
for some sequence $\epsilon_k\to 0$. 
\end{enumerate}
\end{Proposition}
\begin{proof}
	From \cite[p. 26]{SY} we know that there exist a proper smooth function $\rho:M\to \R$ and a constant $C$ such that $|\nabla \rho|_{\omega_0}\leq C$, $\rho(x) \geq Cr_{\omega_0}(x)$ and $|\nabla^2\rho|_{\omega_0}\leq C$ where $r_{\omega_0}(x)$ is the distance function from some point. 
	
	Then, if $\sup_{M\times [0, T]}f(x, t) = \sup_{M}f(x, 0)$, we are done, so suppose  $\sup_{M\times [0, T]}f(x, t) > \sup_{M}f(x, 0)$. Pick points $(x_{k}, t_{k})$ to achieve the maximum of the function $f(x, t) - \frac{1}{k} \rho(x)$, which has a maximum since $f$ is bounded above and $\rho(x) \geq Cr(x)$ goes to infinity at infinity. Then it is easy to see that $t_k\neq 0$ for $k$ large, so we have
\[\lim_{k\to \infty}f(x_k, t_k)  = \sup_{M\times [0, T]} f \qquad\qquad |\nabla f|(x_k, t_k) = \frac{1}{k}|\nabla \rho|_{g(t_k)}(x_k)\]
\[(\dd{t}-\lapl)f(x_k, t_k) \geq \frac{1}{k}(\dd{t}-\lapl)\rho(x_k) = -\frac{1}{k}\lapl_{g(t_k)} \rho(x_k)\]
since the metrics are uniformly equivalent along the Ricci flow, then $|\nabla \rho|_{g_0}\leq C$ implies $|\nabla \rho|_{g(t)}\leq C$. We only need to show that $|\lapl_{g(t)} \rho |\leq C$ 
\begin{align*}
|\nabla_{g(t)}^2\rho-\nabla_{g(0)}^2\rho| &= |(\Gamma_{g(t)}-\Gamma_{g(0)})_{ij}^k\d_k \rho)|  \leq C|\Gamma_{g(t)}-\Gamma_{g(0)}| \leq C\int_{0}^t |\dot{\Gamma}|(s)\,ds \\
&\leq C\int_{0}^t |\nabla\Ric(g(s))|\, ds \leq C\int_{0}^t\frac{1}{s^{1/2}}\,ds\leq C
\end{align*}
and since $|\nabla_{g(0)}^2\rho|\leq C$, we conclude that $|\nabla_{g(t)}^2\rho|\leq C$ is always bounded, which means $|\lapl_{g(t)}\rho|\leq C$. 
\end{proof}

Now we prove a uniform lower bounds on the evolving metrics $\omega(t)$. This is an adaptation of the Schwarz lemma calculation of Wu and Yau in \cite{WY1}. The original calculation is based on Yau's generalized Schwarz lemma \cite{YauS} and Royden's improvement \cite{Royden} to the setting of holomorphic sectional curvature. 
\begin{Proposition}
Along the flow, we have
\begin{equation}
\omega_0\leq \max(n, \frac{2n}{(n+1)\kappa})\omega(t)
\end{equation}
for all time $t$, as long as the flow exists. 
\end{Proposition}
\begin{proof}
We compute the evolution of the quantity $S = \Tr_{\omega(t)}\omega_0$, adopting the notation that $\omega_0 =\sqrt{-1}\hat{g}_{i\bj}\,dz^i\wedge d\bar{z}^j$ and $\omega(t) = \sqrt{-1}g_{i\bj}\,dz^i\wedge d\bar{z}^j$, and we put a hat over all quantities and objects associated to the metric $\omega_0$

\begin{align*}
	\left(\dd{t}-\lapl_{\omega(t)}\right)S &= \dd{t}g^{i\bj}\hat{g}_{i\bj} - g^{k\bl}\d_k\d_{\bl}(g^{i\bj}\hat{g}_{i\bj})\\
	& = -g^{i\bl}g^{k\bj}(\dd{t}g_{k\bl})\hat{g}_{i\bj} + g^{k\bl}g^{i\bj}\hat{R}_{k\bl i\bj} -g^{i\bn}g^{m\bj}R_{m\bn}\hat{g}_{i\bj} - g^{k\bl}g_{m\bn}\hat{g}_{i\bj}\hat{\nabla}_kg^{i\bn}\hat{\nabla}_{\bl}g^{m\bj}\\
	& = g^{i\bl}g^{k\bj}g_{k\bl}\hat{g}_{i\bj} + g^{k\bl}g^{i\bj}\hat{R}_{k\bl i\bj} - g^{k\bl}g_{m\bn}\hat{g}_{i\bj}\hat{\nabla}_kg^{i\bn}\hat{\nabla}_{\bl}g^{m\bj}\\
	& = S + g^{k\bl}g^{i\bj}\hat{R}_{k\bl i\bj} - g^{k\bl}g_{m\bn}\hat{g}_{i\bj}\hat{\nabla}_kg^{i\bn}\hat{\nabla}_{\bl}g^{m\bj}\\
\end{align*}
By Royden's lemma. \cite[pg. 552]{Royden} we have
\[g^{k\bl}g^{i\bj}\hat{R}_{k\bl i\bj} \leq  - \frac{n+1}{2n}\kappa S^2\]
so
\begin{equation}
	\left(\dd{t}-\lapl_{\omega(t)}\right)S \leq S - \frac{n+1}{2n}\kappa S^2 - g^{k\bl}g_{m\bn}\hat{g}_{i\bj}\hat{\nabla}_kg^{i\bn}\hat{\nabla}_{\bl}g^{m\bj}
\end{equation}
and it follows
\begin{align*}
\left(\frac{\d}{\d t}-\lapl_{\omega(t)}\right)\log S & = \frac{\left(\frac{\d}{\d t}-\lapl_{\omega(t)}\right)S }{S}+\frac{|\nabla S|_{g^2}}{S^2} \\
&\leq 1-\frac{n+1}{2n}\kappa S +\frac{|\nabla S|_{g^2}}{S^2}-\frac{g^{k\bl}g_{m\bn}\hat{g}_{i\bj}\hat{\nabla}_k g^{i\bl}\hat{\nabla}_{\bl}g^{m\bj}}{S}
\end{align*}
we need the following algebraic identity, which was one of the steps in Yau's proof of the Calabi conjecture. The proof can be found in \cite[p.349]{Yau}
\[\frac{|\nabla S|_{g^2}}{S^2}-\frac{g^{k\bl}g_{m\bn}\hat{g}_{i\bj}\hat{\nabla}_k g^{i\bl}\hat{\nabla}_{\bl}g^{m\bj}}{S}\leq 0\]
we obtain the following differential inequality for $S$. 
\begin{equation}
\left(\frac{\d}{\d t}-\lapl_{\omega(t)}\right)\log S \leq 1-\frac{n+1}{2n}\kappa S
\end{equation}
Since $S$ is bounded above by our assumption, we can apply the maximum principle to this inequality and get a uniform estimate for $S$ up to any time, 
\[S\leq \max(n, \frac{2n}{(n+1)\kappa})\]
and this implies 
\[\omega_0\leq \max(n, \frac{2n}{(n+1)\kappa}) \omega(t)\]
giving the desired lower bound for the metrics $\omega(t)$. 
\end{proof}

\section{Upper bound for the metric}
In this section we will provide an upper bound for the metric. To do this, we need to define a potential $\varphi$ and write the K\"ahler-Ricci flow equation as an equation of the potential $\varphi$. Define
\begin{equation}
\varphi(x, t) := e^{-t}\int_0^te^{s}\log \frac{\omega(s)^n}{\omega_0^n}ds
\end{equation}
differentiating this equation gives
\begin{equation}\label{eq:phi}
\dd{t}\varphi = \log \frac{\omega(t)^n}{\omega_0^n}-\varphi
\end{equation}
Then if we let $\tilde{\omega}(t) = e^{-t}(\omega_0 - i\d\db \log \omega_0^n)+i\d\db \log \omega_0^n +i\d\db \varphi$, by \eqref{eq:phi}, $\tilde{\omega}$ satisfies the ODE, 
\[\dd{t}\tilde{\omega}(t) = -\Ric(\omega(t))-\tilde{\omega}(t)\]
and since $\tilde{\omega}(0) = \omega_0$, we can conclude from uniquess of solutions to ODE that $\tilde{\omega}(t) = \omega(t)$. Hence the potential $\varphi$ satisfies the equation
\begin{equation}\label{eq:CMA}
\dd{t}\varphi = \log \frac{(e^{-t}(\omega_0-i\d\db \log \omega_0^n)+ i\d \db \log \omega_0^n+i\d\db \varphi)^n}{\omega_0^n}-\varphi
\end{equation}

\begin{Proposition}\label{decayofphidot}
There exist a constant $C>0$ depending only on $B$ and $\kappa$ in theorem \ref{main}, such that
\[|\varphi|\leq C \qquad\text{ and } \qquad|\dot{\varphi}|\leq Cte^{-t}\]
for any $t$. 
\end{Proposition}
\begin{proof}
We compute the evolution of $\dot{\varphi}$, 
\begin{align*}
\dd{t}\dot{\varphi} &= \dd{t}\left(\log \frac{\omega^n(t)}{\omega_0^n}-\varphi\right) = \tilde{g}^{i\bj}\dd{t}\tilde{g}_{i\bj}-\dot{\varphi}\\
& = \Tr_{\omega(t)}\left[\dd{t}(e^{-t}(\omega_0-i\d\db\log \omega_0^n)+i\d\db\log \omega_0^n+i\d\db\varphi)\right]-\dot{\varphi}\\
& = \lapl_{\omega(t)}\dot{\varphi}-e^{-t}\Tr_{\omega(t)}(\omega_0-i\d\db\log\omega_0^n)-\dot{\varphi}
\end{align*}
we can rewrite this as
\begin{equation}\label{eq:evolution}
\left(\dd{t}-\lapl_{\omega(t)}\right)(e^t\dot{\varphi}) = -\Tr_{\omega(t)}(\omega_0+ \Ric(\omega_0))
\end{equation}
by the previous section we know that $\Tr_{\omega(t)}\omega_0$ is bounded. 
Also by our assumptions, we know that the curvature tensor at the initial time is bounded so $|\Ric(\omega_0)|_{\omega_0}$ and $\Tr_{\omega(t)}\omega_0$ are both bounded, which gives a bound
\[|\Tr_{\omega(t)}\Ric(\omega_0)|\leq C\]
so the right hand side of \eqref{eq:evolution} is bounded 
\begin{equation}
\left|\left(\dd{t}-\lapl_{\omega(t)}\right)(e^t\dot{\varphi})\right| \leq C
\end{equation}
we can apply the maximum principle to this and get
\[|e^t\dot{\varphi}|\leq Ct\implies |\dot{\varphi}|\leq Cte^{-t}\]
then integrating this gives the bound on the potential $\varphi$. 
\end{proof}

\begin{Corollary}\label{equivalenceOfMetrics}
	There exist a constant $C>0$, depending only on the constants $B$ and $\kappa$ in theorem \ref{main}, such that
	\[C^{-1}\omega_0\leq \omega(t)\leq C\omega_0\]
	for any $t$, as long as the flow exists. 
\end{Corollary}
\begin{proof}
From the equation \eqref{eq:phi}, we know that
\[\omega(t)^n = e^{\varphi+\dot{\varphi}}\omega_0^n\]
and by our proposition, we know that $\varphi+\dot{\varphi}$ is bounded for all time, hence we have
\[\omega(t)^n\leq C\omega_0^n\]
and from the previous section, we also have
\[\omega_0\leq C\omega(t)\]
If we choose an orthonormal frame on the tangent space such that $(g_0)_{i\bj} = \delta_{ij}$ and $g_{i\bj} = \lambda_i \delta_{ij}$, then the inequalities can be written as
\[\prod_{i=1}^n \lambda_i\leq C\]
and 
\[ \frac{1}{\lambda_i}\leq C\]
so combining them we get
\[\lambda_j = \frac{\prod_{i=1}^n\lambda_i}{\prod_{i\neq j}\lambda_i} \leq C^n\]
this gives the upper bound for $\lambda_i$, which gives an upper bound for $\omega(t)$. 
\end{proof}
From this, the long time existence and convergence can be obtained from standard estimates of the K\"ahler-Ricci flow. For example, we can apply local estimates in \cite{ShW} applied to the holomorphic coordinate systems constructed in \cite{WY3} to obtain the result. For the sake of completeness, we will provide an argument here, the argument is elementary in that it uses only the maximum principle.  

\section{Long Time Existence}
In this section, we show the flow exists for all time and the curvature and its covariant derivatives are all bounded. The method is the same as in the compact case, see for example \cite{SW}. 

\begin{Proposition}\label{LTE}
The flow $(M, g(t))$ exists for all time, and there exist constants $C_l$ depending only on the constants $B$ and $\kappa$ from theorem \ref{main} such that
\[\sup_M|\nabla_{g(t)}^l\Rm(g(t))|_{g(t)}\leq C_l\]
for any $t\in [1, \infty)$. 
\end{Proposition}

\begin{proof}
Suppose the flow exists on a maximal time interval $[0, T_{\max})$, it suffices for us to bound the curvature of $\omega(t)$ are bounded up to the final time $T_{\max}$. Then by Shi's theorem, we can then extend the flow past time $T_{\max}$ and we get our desired bounds on the covariant derivatives of curvature. 

First we need an estimate on the derivative of the metric, such estimates were used by Yau in his solution of the Calabi conjecture \cite{Yau}. 

Consider the following quantity
\[S = |\hat{\nabla}g|_{g}^2 = g^{i\bl}g^{n\bj}g^{k\bm}\hat{\nabla}_ig_{\bj k}\hat{\nabla}_{\bl}g_{\bm n} = |T|_g^2\]
where $\hat{\nabla}$ is the connection with respect to some fixed reference metric $\hat{g}$, and $T_{ik}^l = \Gamma_{ik}^l-\hat{\Gamma}_{ik}^l$ is the difference of the connections. $S$ satisfy the following identity
\begin{align}
(\dd{t}-\lapl)S &= -|\nabla T|^2-|\overline{\nabla}T|^2+S -2\Re(\langle g^{a\bar{b}}\nabla_a\hat{R}_{i\bar{b}k}^l,  T_{ik}^l\rangle)\nonumber \\ 
& =  -|\nabla T|^2-|\d_{\bk}T_{ij}^l|^2+S -\hat{\nabla}\hat{R}\star T-\hat{R}\star T\star T\nonumber\\
\label{S}
& = -|\nabla T|^2-|\Rm(g(t))-\Rm(\hat{g})|^2+S -\hat{\nabla}\hat{R}\star T-\hat{R}\star T\star T
\end{align}
and 
\begin{align}
\label{tr}
(\dd{t}-\lapl)\Tr_{\hat{g}}g &= -\Tr_{\hat{g}}g -g^{k\bl}\hat{R}_{k\bl}^{i\bj}g_{i\bj}-g^{n\bm}g^{k\bl}\hat{g}^{i\bj}\hat{\nabla}_kg_{i\bm}\hat{\nabla}_{\bl}g_{n\bj}
\end{align}
see \cite{PSS} or \cite{SW} for the calculations. 
By Shi's theorem, we know the flow must exist for at least some definite amount of time $t\in [0, \epsilon)$ where $\epsilon$ depends only on the constant $B$ in theorem \ref{main}. So we can now fix $\hat{g} = g(\epsilon/2)$ for instance, then by Shi's theorem, the curvature of $\hat{g}$ and its covariant derivatives are all bounded by constants depending only on the initial curvature bound $B$. 

Then the identity \eqref{S} gives 
\[(\dd{t}-\lapl)S \leq S-\hat{\nabla}\hat{R}\star T-\hat{R}\star T\star T\leq  C(S+1)\]
and from \eqref{tr}, 
\[(\dd{t}-\lapl)\Tr_{\hat{g}}g \leq M-M^{-1}S\]
so we have
\[(\dd{t}-\lapl)(S+A\Tr_{\hat{g}}g)\leq (C-AM^{-1})S + MA+C\]
Letting $A = 2CM$ and applying the maximum principle, we obtain that $S\leq C$ for some uniform constant $C$. 

Now we bound the curvature, the norm of the curvature tensor satisfies an inequality, \cite[Lemma 2.12]{SW}
\begin{align}
(\dd{t}-\lapl)|\Rm| \leq C(|\Rm|^2+1)
\end{align}
for $C$ depending only on dimension.  And from \eqref{S}, 
\[(\dd{t}-\lapl)S\leq D-|\Rm(g(t))|^2\] 
so we have
\begin{align}
(\dd{t}-\lapl)(|\Rm|+NS)&\leq C(|\Rm|^2+1) + N(D-|\Rm(g)|^2)\\
& = (C-N)|\Rm(g)|^2+C+ND
\end{align}
picking $N=2C$. and applying the maximum principle gives the desired bound on cuvature. Then the bounds on covariant derivatives of curvature follow by Shi's estimates. 
\end{proof}
\section{Convergence}
\begin{Proposition}
There exist constants $C_l'>0$ depending only on $B$ and $\kappa$ from theorem \ref{main} such that 
	\[\left|g(t)\right|_{C^l(M, g(1))}\leq C_l'\]
for all $t\in [1, \infty)$. 
\end{Proposition}
\begin{proof}
	The proof follows a similar method as the one used to bound the derivative of the metric under the K\"ahler-Ricci flow in \cite{PSS}. We will denote $\hat{g} = g(1)$ and put a hat on all quantities associated to $\hat{g}$. Consider the quantities
	\[T_{ij}^k = \Gamma_{ij}^k - \hat{\Gamma}_{ij}^k\]
and define
	\[(M_l)_{i_1\ldots i_{l-1} ij}^k = \nabla^{l-1}_{i_1 \ldots i_{l-1}}T_{ij}^k\]
	and
	\[S_l = |M_l|^2_{g} \]
we would like to obtain bounds for all the $S_l$, this is the content of the following lemma. 
\begin{Lemma}
	For every $l\geq 0$, there exist constants $D_l$ so that the quantities $S_l$ satisfies the following inequality
	\begin{equation}
		(\dd{t}-\lapl)S_l\leq -S_{l+1} + D_l(S_l+1)
	\end{equation}
	the constant $D_l>0$ depend only on $C_l$ from proposition \ref{LTE} and $C$ in corollary \ref{equivalenceOfMetrics}. Then there exist constant $B_l\geq 0$ depending only on $D_l$ such that 
	\[S_l\leq B_l. \]
for all $t\in [1, \infty)$
\end{Lemma}
\begin{proof}
We prove this by induction on $l$, for $l=0$, this is true by the calculation in \cite{PSS}. Suppose the lemma is true for any $l < k$, then we compute

\begin{align}
	(\dd{t}-\lapl)S_k & = \sum_{l=1}^{k-1} g^{i_1 \bar{j}_1}\ldots R^{i_l \bar{j}_l}\cdots g^{i_{k-1} \bar{j}_{k-1}}g^{p\bq}g^{r\bar{s}}g_{m\bn} (M_k)_{i_1\ldots i_{k-1}pr}^m\overline{(M_k)_{j_1\ldots j_{k-1}qs}^n}\nonumber \\
	& \qquad + g^{i_1\bar{j_1}}\ldots g^{i_{k-1}\bar{j}_{k-1}}R^{p\bq}g^{r\bar{s}}g_{m\bn}(M_k)_{i_1\ldots i_{k-1}pr}^m\overline{(M_k)_{j_1\ldots j_{k-1}qs}^n}\nonumber\\
	& \qquad + g^{i_1\bar{j_1}}\ldots g^{i_{k-1}\bar{j}_{k-1}}g^{p\bq}R^{r\bar{s}}g_{m\bn}(M_k)_{i_1\ldots i_{k-1}pr}^m\overline{(M_k)_{j_1\ldots j_{k-1}qs}^n}\nonumber\\
	& \qquad - g^{i_1\bar{j_1}}\ldots g^{i_{k-1}\bar{j}_{k-1}}g^{p\bq}g^{r\bar{s}}R_{m\bn}(M_k)_{i_1\ldots i_{k-1}pr}^m\overline{(M_k)_{j_1\ldots j_{k-1}qs}^n}\nonumber\\
	& \qquad +kS_k+ \langle\dd{t}M_k, M_k\rangle + \langle M_k, \dd{t}M_k\rangle - |\nabla M_k|^2 - |\overline{\nabla} M_k|^2 \nonumber \\
	& \qquad  - \langle \overline{\lapl}M_k, M_k\rangle - \langle M_k, \lapl M_k\rangle\\
\end{align}
where $\lapl = g^{a\bar{b}}\nabla_{\bar{b}}\nabla_{a}$ and $\overline{\lapl} = g^{a\bar{b}}\nabla_a\nabla_{\bar{b}}$ is its conjugate. Then we have 
\begin{align}
	\lapl M_k &= \overline{\lapl}M_k + g^{a\bar{b}}[\nabla_{\bar{b}}, \nabla_a]M_k\\
	& = \overline{\lapl}M_k - \sum_{l=1}^{k-1} R_{i_l}^c(M_k)_{i_1\ldots i_{l-1} c i_{l+1}\ldots i_{k-1} pr}^m \nonumber \\
	&\qquad  +R_{p}^c(M_k)_{i_1\ldots i_{k-1} cr}^m +R_{p}^c(M_k)_{i_1\ldots i_{k-1} pc}^m - R_{c}^m(M_k)_{i_1\ldots i_{k-1} pr}^c
\end{align}
and in particular we have 
\begin{align}
	\langle M_k, \lapl M_k\rangle &= \langle M_k, \overline{\lapl} M_k\rangle + \sum_{l=1}^{k-1} g^{i_1 \bar{j}_1}\ldots R^{i_l \bar{j}_l}\cdots g^{i_{k-1} \bar{j}_{k-1}}g^{p\bq}g^{r\bar{s}}g_{m\bn} (M_k)_{i_1\ldots i_{k-1}pr}^m\overline{(M_k)_{j_1\ldots j_{k-1}qs}^n}\nonumber \\
	& \qquad + g^{i_1\bar{j_1}}\ldots g^{i_{k-1}\bar{j}_{k-1}}R^{p\bq}g^{r\bar{s}}g_{m\bn}(M_k)_{i_1\ldots i_{k-1}pr}^m\overline{(M_k)_{j_1\ldots j_{k-1}qs}^n}\nonumber\\
	& \qquad + g^{i_1\bar{j_1}}\ldots g^{i_{k-1}\bar{j}_{k-1}}g^{p\bq}R^{r\bar{s}}g_{m\bn}(M_k)_{i_1\ldots i_{k-1}pr}^m\overline{(M_k)_{j_1\ldots j_{k-1}qs}^n}\nonumber\\
	& \qquad - g^{i_1\bar{j_1}}\ldots g^{i_{k-1}\bar{j}_{k-1}}g^{p\bq}g^{r\bar{s}}R_{m\bn}(M_k)_{i_1\ldots i_{k-1}pr}^m\overline{(M_k)_{j_1\ldots j_{k-1}qs}^n}
\end{align}
so combining the computations above, we have
\begin{align}
	(\dd{t}-\lapl)S_k  &= -|\nabla M_k|^2-|\overline{\nabla}M_k|^2-2\Re \langle(\dd{r}-\overline{\lapl})M_k, M_k\rangle+kS_k\\
	&\leq -S_{k+1} - 2 \Re\langle(\dd{r}-\overline{\lapl})M_k, M_k\rangle+kS_k\label{S_k}
\end{align}
next we compute, 
\begin{align}
	(\dd{t}-\overline{\lapl})M_k &= (\dd{t}-\overline{\lapl})\nabla^{k-1}T \\
	&= \nabla^{k-1}(\dd{t}-\overline{\lapl})T + R\star \nabla^{k-1}T+ \sum_{j=1}^{k-1}\nabla^j R \star \nabla^{k-1-j}T\\
	&= -\nabla^{k-1}(g^{a\bar{b}}\nabla_{a}\hat{R}_{i\bar{b}p}^l) + R\star \nabla^{k-1}T+ \sum_{j=1}^{k-1}\nabla^j R \star \nabla^{k-1-j}T\\
	&= -\nabla^{k-1}(\hat{\nabla}^{\bar{b}}\hat{R}_{i\bar{j}p}^)+\nabla^{k-1}(T\star\hat{R})+ R\star \nabla^{k-1}T+ \sum_{j=1}^{k-1}\nabla^j R \star \nabla^{k-1-j}T\\
	& =\hat{R}\star\nabla^{k-1}T+ R\star \nabla^{k-1}T + Q(T, \ldots, \nabla^{k-2}T, R, \ldots, \nabla^{k-1}R, \hat{R}, \ldots, \hat{\nabla}^k \hat{R})
\end{align}
where $Q$ is some expression involving only the quantities in the bracket, in particular this term is bounded by the induction hypothesis. So by the induction hypothesis, we have 
\begin{equation}
	|(\dd{t}-\lapl)M_k|\leq C(|\nabla^{k-1}T|+1) = C(|M_k|+1)
\end{equation}
then from \eqref{S_k}, we get
\begin{align}
(\dd{t}-\lapl)S_k &\leq -S_{k+1} + C_k(S_k + \sqrt{S_k})\\
&\leq -S_{k+1}+C_k(S_k+1)
\end{align}
this proves the first part of the lemma. Next, from this we have
\[(\dd{t}-\lapl)(S_k + AS_{k-1})\leq (C_k-A)S_k + C_k+AC_{k-1}(S_{k-1}+1)\]
and by choosing $A = C_k+1$ and using the induction hypothesis that $S_{k-1}$ is bounded, we get 
\[(\dd{t}-\lapl)(S_k + AS_{k-1})\leq -S_k + C_k+AC_{k-1}(B_{k-1}+1)\]
applying the maximum principle, we get the bounds $S_l\leq B_l$, where we can set $B_l$ to be $C_k+AC_{k-1}(B_{k-1}+1)$. 
\end{proof}
So we have uniform bounds on $S_k$, which means bounds on all derivatives of $T$ in the holomorphic directions. But we claim that actually all derivatives of $T$ are bounded. To see this, note that the identity $\nabla_{\bar{i}}T_{jk}^l = \hat{R}_{\bi j k}^l-R_{\bi j k}^l$ allows us to get rid of a derivative in the antiholomorphic direction and replace it by curvature terms, which we know are bounded and have bounded derivatives. So to bound the mixed derivatives, we can first commute the derivatives to move the antiholomorphic ones in front and then use this identity and the fact that the curvature and all its derivatives are bounded. From this we can conclude that $|T|_{C^k(M, g(t))}\leq C_k$ it follows that $|T|_{C^k(M, \hat{g})}\leq C_k$, so we get uniform bounds on the $C^k$ norms of $T$ with respect to the fixed background metric. Then from the following identity
\[T_{ij}^kg_{\bm k} = \hat{\nabla}_{i}g_{\bm j}\]
we can get uniform bounds on the $C^k$ norm of $g$. 
\end{proof}
Now we can prove the main theorem, 
\begin{proof}[Proof of theorem \ref{main}]
By \eqref{decayofphidot}, we know that $\varphi(t)$ convergences in $C^0$ to some limit $\varphi_{\infty}$, which is continuous, and $\dot{\varphi}$ converges to 0 uniformly. Now from our definition of $\varphi$, we have
\[\varphi(x, t) := e^{-t}\int_0^te^{s}\log \frac{\omega(s)^n}{\omega_0^n}ds\]
and since $\omega(s)$ is bounded in $C^k$ norm by our previous lemma, so the right hand side is bounded in $C^k$ norm and we get uniform bounds of $|\varphi |_{C^k(M, g_0)}$ for all $k$. Hence the convergence of $\varphi$ to $\varphi_{\infty}$ is actually in $C^{\infty}$. So we can take a limit of the following equation
\[\dot{\varphi} = \log \frac{[e^{-t}(\omega_0-i\d\db \log \omega_0^n)+i\d\db \log \omega_0^n + i\d\db\varphi]^n}{\omega_0^n}-\varphi\]
and we get
\[0 = \log \frac{(i\d\db\log \omega_0^n + i\d\db\varphi_{\infty})^n}{\omega_0^n}-\varphi_{\infty}\]
hence $\omega_{\infty} = i\d\db \log \omega_0^n + i\d\db \varphi_{\infty}$ is K\"ahler-Einstein and the bounds on the curvature and its covariant derivatives follows from the $C^{\infty}$ convergence of the metrics. 
\end{proof}
\paragraph{Acknowledgements:}I would like to thank my advisor Duong Phong for suggesting the question and for his advice and guidence. 

\Addresses
\end{document}